\documentclass[11pt, reqno]{amsart}
\usepackage[margin=1in]{geometry}
\usepackage[style=numeric-comp]{biblatex}

\usepackage{amsmath}
\usepackage{amssymb}
\usepackage{enumerate}
\usepackage{tikz}
\usepackage{subcaption}
\usepackage{mathtools}

\usetikzlibrary{calc}
\usetikzlibrary{intersections}

\newtheorem{theorem}{Theorem} [section]
\newtheorem{lemma}[theorem]{Lemma}

\DeclareMathOperator{\cl}{cl}

\title{A new minor-closed class of transversal matroids}

\author[G.\ Toft]{Gerry Toft}
\email{gerrytoft3@gmail.com}

\date{\today}


\begin{document}
\begin{abstract}
We provide a characterisation of when a single-element contraction of a transversal matroid is itself transversal. Using this characterisation, we define a new class of transversal matroids closed under minors, which we call path-circular matroids. Path-circular matroids generalise both of the well-known classes of bicircular matroids and multi-path matroids.
\end{abstract}

\maketitle

\section{Introduction}

Along with representable and graphic matroids, transversal matroids are one of the most common examples of naturally-arising matroids. With this in mind, it is perhaps surprising that many aspects of the class of transversal matroids remain mysterious. Unlike representable and graphic matroids, transversal matroids are not closed under minors, and it is not known when the contraction of a transversal matroid is itself transversal.

A recent breakthrough concerning minors of transversal matroids was made by Bastida \cite{bas24}, who found a polynomial-time algorithm to determine whether a single-element contraction of a transversal matroid $M$ is  also transversal, and to construct a presentation if so. We consider the same problem from a structural, rather than algorithmic, point of view. The aim is to provide a tool which can be used in proofs concerning minors of transversal matroids. The first main results of this paper are Theorems~\ref{thm: contraction} and \ref{thm: construction}. Theorem~\ref{thm: contraction} is a characterisation for when a single-element contraction of a transversal matroid is transversal, and, if it is, Theorem~\ref{thm: construction} describes how to find a presentation. These theorems will require the following definitions.

Let $M$ be a transversal matroid with presentation $\mathcal A = (A_1, A_2, \ldots, A_r)$, and let $e \in E(M)$. A graph $G$ is an \emph{$(e, \mathcal A)$-presenting graph} if there is a bijection 
$$
\phi : V(G) \rightarrow \{i \in \{1,2,\ldots,r\} : e \in A_i\}
$$ 
such that, for all distinct $v_1, v_2 \in V(G)$, the subgraph of $G$ induced by the vertices
$$
\{u \in V(G) : A_{\phi(u)} \subseteq \cl_M^*(A_{\phi(v_1)} \cup A_{\phi(v_2)})\}
$$
is connected. We call the bijection $\phi$ an \emph{$(e, \mathcal A)$-presenting map of $G$}.

Furthermore, $G$ is a \emph{minimal $(e, \mathcal A)$-presenting graph} if it is an $(e, \mathcal A)$-presenting graph, and the deletion of each edge of $G$ is not an $(e, \mathcal A)$-presenting graph. For instance, Figure~\ref{fig: min example} shows minimal $(e, \mathcal A_i)$-presenting graphs for
\begin{align*}
\mathcal A_1 &= (\{e,u,v\},\{e,w,x\},\{w,y,z\}) \,, \\
\mathcal A_2 &= (\{e,s,t,u,v\},\{e,u,v,w,x\},\{e,w,x,y,z\}) \,, \text{ and} \\
\mathcal A_3 &= (\{e,w,x,y,z\},\{e,w,x,y,z\},\{e,w,x,y,z\},\{e,w,x,y,z\}) \,.
\end{align*}
There is a unique minimal $(e, \mathcal A_1)$-presenting graph, and a unique minimal $(e, \mathcal A_2)$-presenting graph, while any tree on $4$ vertices is a minimal $(e, \mathcal A_3)$-presenting graph.

\begin{figure}
\centering
\begin{subfigure}[t]{0.31\textwidth}
\centering
\begin{tikzpicture}
\coordinate (a1) at (0,0);
\coordinate (a2) at (-1,-1.73);
\coordinate (a3) at (1,-1.73);

\filldraw (a1) circle (3pt) node[label=above:$\{e{,}u{,}v\}$] {};
\filldraw (a2) circle (3pt) node[label=below:$\{e{,}w{,}x\}$,xshift=-5] {};
\filldraw (a3) circle (3pt) node[label=below:$\{e{,}y{,}z\}$,xshift=5] {};

\draw (a1) -- (a2) -- (a3) -- (a1);
\end{tikzpicture}
\subcaption{A minimal $(e, \mathcal A_1)$-presenting graph.}
\end{subfigure}
\hspace{0.02\textwidth}
\begin{subfigure}[t]{0.31\textwidth}
\centering
\begin{tikzpicture}
\coordinate (a1) at (-1.5,0);
\coordinate (a2) at (0,0);
\coordinate (a3) at (1.5, 0);

\filldraw (a1) circle (3pt) node[label=above:$\{e{,}s{,}t{,}u{,}v\}$] {};
\filldraw (a2) circle (3pt) node[label=below:$\{e{,}u{,}v{,}w{,}x\}$] {};
\filldraw (a3) circle (3pt) node[label=above:$\{e{,}w{,}x{,}y{,}z\}$] {};

\filldraw[white] (0,-1.5);

\draw (a1) -- (a3);
\end{tikzpicture}
\subcaption{A minimal $(e, \mathcal A_2)$-presenting graph.}
\end{subfigure}
\hspace{0.02\textwidth}
\begin{subfigure}[t]{0.31\textwidth}
\centering
\begin{tikzpicture}
\coordinate (a1) at (0,0);
\coordinate (a2) at (-1, -1.73);
\coordinate (a3) at (1, -1.73);
\coordinate (a4) at (0, -1.15);

\filldraw (a1) circle (3pt) node[label=above:$\{e{,}w{,}x{,}y{,}z\}$] {};
\filldraw (a2) circle (3pt) node[label=below:$\{e{,}w{,}x{,}y{,}z\}$,xshift=-10] {};
\filldraw (a3) circle (3pt) node[label=below:$\{e{,}w{,}x{,}y{,}z\}$,xshift=10] {};
\filldraw (a4) circle (3pt) node[label=right:$\{e{,}w{,}x{,}y{,}z\}$,yshift=5,xshift=-2] {};

\draw (a1) -- (a4) -- (a3);
\draw (a2) -- (a4);
\end{tikzpicture}
\subcaption{A minimal $(e, \mathcal A_3)$-presenting graph.}
\end{subfigure}
\caption{Minimal $(e, \mathcal A_i)$-presenting graphs.} \label{fig: min example}
\end{figure}
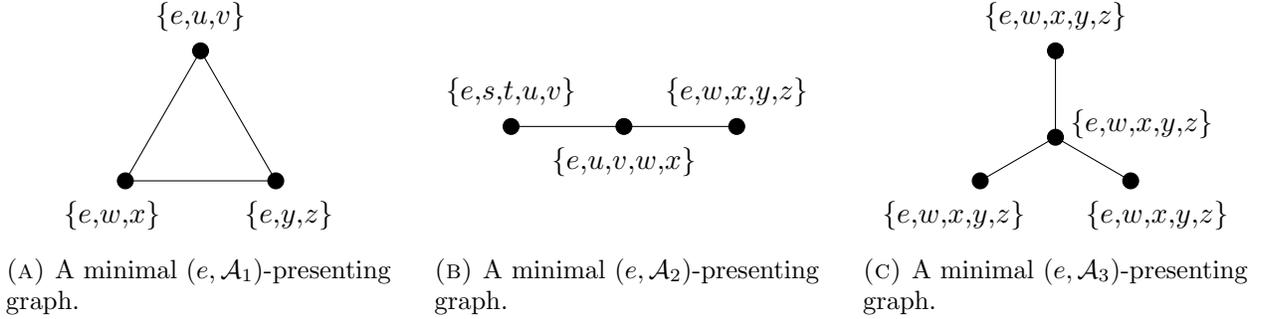

\begin{theorem} \label{thm: contraction}
Let $M$ be a transversal matroid with presentation $\mathcal A$ such that $|\mathcal A| = r(M)$, and let $e \in E(M)$. Let $G$ be a minimal $(e, \mathcal A)$-presenting graph. Then $M / e$ is transversal if and only if $G$ is a tree.
\end{theorem}

\begin{theorem} \label{thm: construction}
Let $M$ be a transversal matroid, and let $e \in E(M)$. Let $\mathcal A = (A_1,A_2,\ldots,A_r)$ be a presentation of $M$, with $r = r(M)$, such that $e \in A_i$ for all $i \leq k$, and $e \notin A_j$ for all $j > k$. Let $G$ be a minimal $(e, \mathcal A)$-presenting graph which is a tree, and let $\phi$ be an $(e, \mathcal A)$-presenting map of $G$. Suppose the edge set of $G$ is
$$
E(G) = \{\{u_1,v_1\},\{u_2,v_2\},\ldots,\{u_{k-1},v_{k-1}\}\} \,.
$$
Then
$$
(A_{\phi(u_1)} \cup A_{\phi(v_1)} - e, A_{\phi(u_2)} \cup A_{\phi(v_2)} - e, \ldots, A_{\phi(u_{k-1})} \cup A_{\phi(v_{k-1})} - e, A_{k+1}, A_{k+2}, \ldots, A_r)
$$
is a presentation of $M / e$.
\end{theorem}

This research was motivated by the problem of finding the maximal minor-closed class of transversal matroids. The largest known minor-closed classes of transversal matroids are \emph{bicircular matroids} and \emph{multi-path matroids}. Bicircular matroids were introduced by Sim\~oes-Pereira \cite{sim72} as a matroid whose elements are the edges of a graph, and whose circuits are minimal, connected subgraphs containing at least two cycles. Matthews \cite{mat77} showed that a bicircular matroid can also be viewed as a transversal matroid that has a presentation in which every element is contained in at most two sets. To construct such a presentation $\mathcal A$, add a set to $\mathcal A$, for each vertex of the graph, consisting of the edges incident to that vertex. Multi-path matroids, introduced by Bonin and Gim\'enez \cite{bon07}, are transversal matroids with a cyclic ordering $\sigma$ of their elements, and a presentation $\mathcal B$ such that each set of $\mathcal B$ is a cyclic interval of $\sigma$, and no set of $\mathcal B$ contains another.

Theorem~\ref{thm: contraction} implies that a necessary condition for a transversal matroid $M[\mathcal A]$ to be a member of the maximal minor-closed class of transversal matroids is that, for all $e \in E(M[\mathcal A])$, each minimal $(e, \mathcal A)$-presenting graph is a tree. We define a class of transversal matroids which satisfy a stronger condition, that these graphs are paths, and show that this class is closed under minors.

Let $G$ be a simple graph. A \emph{path-circular collection $\mathcal P$ of $G$} is a collection of (not necessarily distinct) paths of $G$ such that, if $u_1,u_2,\ldots,u_k$ are the vertices (in order) of a path in $\mathcal P$ and $i \in \{2,3,\ldots,k-1\}$, then
\begin{enumerate}[(i)]
	\item the degree of $u_i$ in $G$ is $2$, and
	\item every path of $\mathcal P$ which contains $u_i$ also contains at least one of $u_1$ and $u_k$.
\end{enumerate}
Note that a path-circular collection may contain ``null paths" consisting of zero vertices. These will correspond to loops in the resulting matroid.

Suppose $V(G) = \{v_1,v_2,\ldots,v_r\}$. For all $i \in \{1,2,\ldots,r\}$, let $N(v_i) = \{p \in \mathcal P : v_i \text{ is a vertex of } p\}$. Denote by $M(\mathcal P)$ the transversal matroid on groundset $\mathcal P$ with presentation $(N(v_1),N(v_2),\ldots,N(v_r))$. We call such a matroid a \emph{path-circular matroid}.

The definition of path-circular matroids is a generalisation of the definition of bicircular matroids: the elements of a bicircular matroid correspond to edges of a graph, while the elements of a path-circular matroid correspond to paths. Thus, a bicircular matroid is a path-circular matroid in which every path of $\mathcal P$ contains at most one edge. Note that, in this case, the conditions (i) and (ii) of a path-circular collection are satisfied trivially. Path-circular matroids also generalise multi-path matroids: a multipath matroid is a path-circular matroid in which the graph $G$ is a cycle. This is illustrated in Figure~\ref{fig: multipath}. In the image of a path-circular matroid in Figure~\ref{fig: multipath}, the paths of $\mathcal P$ are unbroken lines, with loops representing paths consisting of a single vertex, and the edges of the graph $G$ are the non-loop edges implied by the paths.

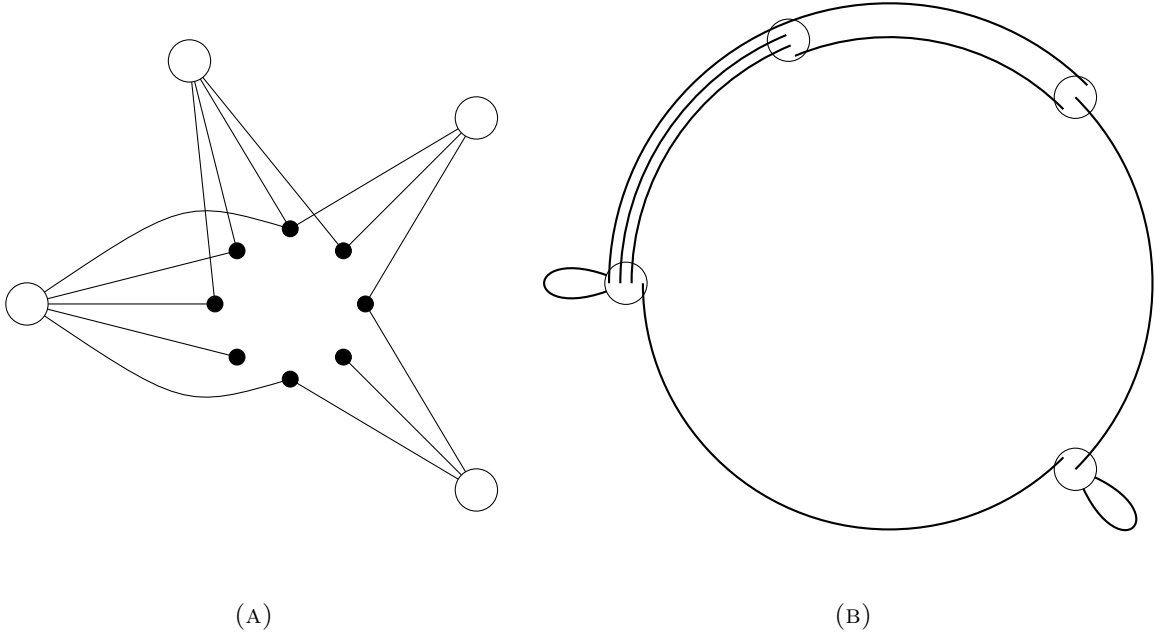
\begin{figure}
\centering
\begin{subfigure}{0.4\textwidth}
\centering
\begin{tikzpicture}
\foreach \theta in {0,...,7}
	\filldraw (\theta*45:1) circle (3pt);
	
\coordinate (a1) at (45:3.5);
\coordinate (a2) at (112.5:3.5);
\coordinate (a3) at (180:3.5);
\coordinate (a4) at (315:3.5);

\draw[white] (3.725,0) arc (0:360:3.725);

\draw (a1) -- (0:1);
\draw (a1) -- (45:1);
\draw (a1) -- (90:1);

\draw (a2) -- (45:1);
\draw (a2) -- (90:1);
\draw (a2) -- (135:1);
\draw (a2) -- (180:1);

\draw (a3) .. controls (135:2) .. (90:1);
\draw (a3) -- (135:1);
\draw (a3) -- (180:1);
\draw (a3) -- (225:1);
\draw (a3) .. controls (225:2) .. (270:1);

\draw (a4) -- (270:1);
\draw (a4) -- (315:1);
\draw (a4) -- (0:1);

\draw[fill = white] (a1) circle (8pt);
\draw[fill = white] (a2) circle (8pt);
\draw[fill = white] (a3) circle (8pt);
\draw[fill = white] (a4) circle (8pt);
\end{tikzpicture}
\subcaption{}
\end{subfigure}
\begin{subfigure}{0.55\textwidth}
\centering
\begin{tikzpicture}
\coordinate (a1) at (45:3.5);
\coordinate (a2) at (112.5:3.5);
\coordinate (a3) at (180:3.5);
\coordinate (a4) at (315:3.5);

\draw[thick] (a3) .. controls (172:5) and (188:5) .. (a3);
\draw[thick] (a4) .. controls (323:5) and (307:5) .. (a4);

\draw[fill = white] (a1) circle (8pt);
\draw[fill = white] (a2) circle (8pt);
\draw[fill = white] (a3) circle (8pt);
\draw[fill = white] (a4) circle (8pt);

\draw[thick] (45:3.725) arc (45:180:3.725);
\draw[thick] (112.5:3.575) arc (112.5:180:3.575);
\draw[thick] (112.5:3.425) arc (112.5:180:3.425);
\draw[thick] (180:3.275) arc (180:315:3.275);
\draw[thick] (315:3.5) arc (315:405:3.5);
\draw[thick] (45:3.275) arc (45:112.5:3.275);
\end{tikzpicture}
\subcaption{}
\end{subfigure}
\caption{A multi-path matroid with ground set given by the black vertices (A), and an isomorphic path-circular matroid (B).} \label{fig: multipath}
\end{figure}

\begin{theorem} \label{thm: minors}
Let $M$ be a path-circular matroid. Then $M \backslash e$ and $M / e$ are both path-circular matroids.
\end{theorem}

Section~\ref{prelims} contains some required preliminaries and basic results concerning transversal matroids and their duals. In Section~\ref{contractions}, we prove Theorem~\ref{thm: contraction} and Theorem~\ref{thm: construction}. Finally, in Section~\ref{minorclosed}, we use Theorem~\ref{thm: construction} to prove Theorem~\ref{thm: minors}, and thus show that the class of path-circular matroids is closed under minors.

\section{Preliminaries and Basic Results} \label{prelims}

Our notation and terminology follows \cite{oxley} unless specified otherwise. We will use $[n]$ to represent the set $\{1,2,\ldots,n\}$.

Let $G$ be a graph, and $U$ a set of vertices of $G$. The \emph{subgraph of $G$ induced by $U$} is the graph obtained from $G$ by deleting every vertex not in $U$. We shall denote a path of $G$ by an ordered list of vertices which the path passes through. If $p = v_1,v_2,\ldots,v_n$ is a path of $G$, then we say the vertices $v_1$ and $v_n$ are the \emph{end vertices} of $p$.

\subsection{Transversal Matroids}

Let $\mathcal A = (A_1, A_2, \ldots, A_r)$ be a collection of (not necessarily distinct) subsets of a set $E$. A \emph{partial transversal} of $\mathcal A$ is a subset $\{x_1, x_2, \ldots, x_m\}$ of $E$ such that the elements $x_1, x_2, \ldots, x_m$ are distinct, and there exist distinct elements $k_1, k_2, \ldots, k_m \in [r]$ with $x_i \in A_{k_i}$ for all $i \in [m]$. A \emph{transversal} of $\mathcal A$ is a partial transversal $X$ such that $|X| = r$.

Edmonds and Fulkerson \cite{edm65} showed that there is a matroid on ground set $E$ whose independent sets are the partial transversals of $\mathcal A$. This matroid is called a \emph{transversal matroid}, and denoted $M[\mathcal A]$. The collection of sets $\mathcal A$ is called a \emph{presentation} of $M[\mathcal A]$.

The following is a well-known theorem due to Hall \cite{hal35} concerning transversals.

\begin{lemma} \label{lem: Hall}
Let $\mathcal A = (A_1, A_2, \ldots, A_r)$. Then $\mathcal A$ has a transversal if and only if, for all $J \subseteq [r]$,
$$
\left|\bigcup_{j \in J} A_j\right| \geq |J| \,.
$$
\end{lemma}

Every deletion of a transversal matroid is transversal. This can be seen from the following (straightforward) lemma.

\begin{lemma} \label{lem: deletion presentation}
Let $M$ be a transversal matroid with presentation $\mathcal A = (A_1, A_2, \ldots, A_r)$, and let $X \subseteq E(M)$. Then 
$$
(A_1 \cap X, A_2 \cap X, \ldots, A_r \cap X)
$$
is a presentation of $M | X$.
\end{lemma}

Although Theorem~\ref{thm: contraction} and Theorem~\ref{thm: construction} are stated in terms of contractions of transversal matroids, we shall prove these theorems by considering deletions of the duals of transversal matroids. We shall often refer to the duals of transversal matroids as \emph{co-transversal} matroids. The following three lemmas, concerning the rank and independent sets of a co-transversal matroid, will be useful. For the first of these, see \cite[Lemma 4.1]{ing73}.

\begin{lemma} \label{lem: rank_general}
Let $M$ be the dual of a transversal matroid with presentation $\mathcal A = (A_1, A_2, \ldots, A_r)$ such that $r = r^*(M)$. Let $X \subseteq E(M)$. Then 
$$
r(X) = \min_{Y \supseteq X}\left(|Y| - \left|\{i \in [r] : A_i \subseteq Y\}\right|\right) \,.
$$
\end{lemma}

The next lemma, which will be used extensively in the proofs of Theorems~\ref{thm: contraction} and \ref{thm: construction}, is an immediate consequence of Lemma~\ref{lem: rank_general} in the case where $X$ is a flat. It can also be seen as the dual of \cite[Corollary 2.5]{bon15}.

\begin{lemma} \label{lem: rank}
Let $M$ be a the dual of a transversal matroid with presentation $\mathcal A = (A_1, A_2, \ldots, A_r)$ such that $r = r^*(M)$. Let $F$ be a flat of $M$. Then $r(F) = |F| - \left|\{i \in [r] : A_i \subseteq F\}\right|$. 
\end{lemma}

\begin{lemma} \label{lem: independent}
Let $M$ be the dual of a transversal matroid with presentation $\mathcal A = (A_1, A_2, \ldots, A_r)$ such that $r = r^*(M)$. Let $X \subseteq E(M)$. Then $X$ is independent if and only if, for all $J \subseteq [r]$, 
$$
\left|X \cap \left(\bigcup_{j \in J} A_j\right)\right| \leq \left|\bigcup_{j \in J} A_j\right| - \left|J\right| \,.
	$$
\end{lemma}

\begin{proof}
A set $X$ is independent in $M$ if and only if $E - X$ is spanning in $M^*$. Since $r = r^*(M)$, this means that there is a transversal of $\mathcal A$ contained in $E - X$, or, equivalently, the collection
$$
(A_1 \cap (E - X), A_2 \cap (E - X), \ldots, A_r \cap (E - X))
$$
has a transversal. By Lemma~\ref{lem: Hall}, this is true if and only if, for all $J \subseteq [r]$,
$$
\left| (E - X) \cap \left( \bigcup_{j \in J} A_j \right) \right | \geq |J| \,.
$$
This inequality is equivalent to
$$
\left|\bigcup_{j \in J} A_j \right| - \left|X \cap \left(\bigcup_{j \in J} A_j \right) \right| \geq |J| \,,
$$
which can be rearranged to give the desired result.
\end{proof}

\subsection{Cyclic Flats}
A \emph{cyclic flat} of a matroid is a flat which is a union of circuits. The set of cyclic flats of a matroid $M$ is denoted $\mathcal Z(M)$.

\begin{lemma} \label{lem: flats}
Let $M$ be the dual of a transversal matroid with presentation $\mathcal A = (A_1, A_2, \ldots, A_r)$. Let $F$ be a cyclic flat of $M$. Then there exists $J \subseteq [r]$ such that $F = \bigcup_{j \in J} A_j$.
\end{lemma}

\begin{proof}
Let $e \in F$. Since $F$ is cyclic, there is a circuit $C_e$ of $M$ such that $e \in C_e$ and $C_e \subseteq F$. By Lemma~\ref{lem: independent}, there exists $J_e \subseteq [r]$ such that 
\begin{equation}
\left|C_e \cap \left(\bigcup_{j \in J_e} A_j\right)\right| > \left|\bigcup_{j \in J_e} A_j \right| - |J_e| \,.
\label{eq: circuit_elements}
\end{equation}
Let $X_e = \bigcup_{j \in J_e} A_j$. The above inequality is still true if $C_e$ is replaced with $C_e \cap X_e$, so $C_e \cap X_e$ is dependent. Therefore, since $C_e$ is a circuit, we have that $C_e = C_e \cap X_e$, so $C_e \subseteq X_e$.

Next, we show that $X_e \subseteq \cl(C_e)$. Let $x \in X_e - C_e$. Since $C_e - e$ is independent, Lemma~\ref{lem: independent} implies that
$$
|(C_e - e) \cap X_e| \leq |X_e| - |J_e| \,,
$$
but, by (\ref{eq: circuit_elements}),
$$
|C_e \cap X_e| > |X_e| - |J_e| \,.
$$ 
Thus, 
$$
|(C_e - e) \cap X_e| = |X_e| - |J_e| \,,
$$
and so
$$
|((C_e - e) \cup \{x\}) \cap X_e| = |X_e| - |J_e| + 1 \,.
$$
Therefore, by Lemma~\ref{lem: independent}, $(C_e - e) \cup \{x\}$ is dependent, which implies that $x \in \cl(C_e)$. Hence, $X_e \subseteq \cl(C_e)$.

Now consider $J = \bigcup_{e \in F} J_e$. Let $X = \bigcup_{j \in J} A_j$. To complete the proof, we show that $F = X$. For all $e \in F$, we have that $e \in C_e \subseteq X_e \subseteq X$, so $F \subseteq X$. Conversely, 
$$
X = \bigcup_{e \in F} X_e \subseteq \bigcup_{e \in F} \cl(C_e) \subseteq \cl(F) = F \,.
$$ 
Therefore, $F = X$, as desired.
\end{proof}

\subsection{Presentations}
A transversal matroid can have many different presentations. The next lemma will be useful to move between different presentations.

\begin{lemma} \label{lem: move between presentation}
Let $M$ be a transversal matroid with presentation $\mathcal A = (A_1, A_2, \ldots, A_r)$. Let $i \in [r]$, and let $B_i \subseteq E(M)$ such that $\cl^*(B_i) = \cl^*(A_i)$. Furthermore, suppose that if $B \subseteq E(M)$ with $B_i \subseteq B$ and $A_i \not\subseteq B$, then $r^*_M(B) < |B| - |\mathcal A(B)|$. Then,
$$
\mathcal A' = (A_1, A_2, \ldots, A_{i-1}, B_i, A_{i+1}, A_{i+2}, \ldots, A_r)
$$
is a presentation of $M$.
\end{lemma}

\begin{proof}
For a set $X \subseteq E(M)$, define $\mathcal A(X) = \{j \in [r] : A_j \subseteq X\}$. Similarly, let $\mathcal A'(X)$ be $\mathcal A(X) \cup \{i\}$ if $B_i \subseteq X$, and $\mathcal A(X) - i$ otherwise.	
	
Now, let $N = M[\mathcal A']$. We show that $N = M$ by showing that $r^*_N(X) = r^*_M(X)$ for all $X \subseteq E(M)$. First, we show that $|X| - |\mathcal A'(X)| \geq r^*_M(X)$. If $X$ contains neither $A_i$ nor $B_i$, or both $A_i$ and $B_i$, then, by Lemma~\ref{lem: rank_general},
\begin{align*}
	|X| - |\mathcal A'(X)| &= |X| - |\mathcal A(X)| \\
	&\geq r^*_M(X) \,.
\end{align*}
If $A_i \subseteq X$ and $B_i \not\subseteq X$, then
\begin{align*}
	|X| - |\mathcal A'(X)| &= |X| - |\mathcal A(X)| + 1 \\
	&> r^*_M(X) \,.
\end{align*}
Finally, if $A_i \not\subseteq X$ and $B_i \subseteq X$, then $r^*_M(X) < |X| - |\mathcal A(X)|$. Thus,
\begin{align*}
	|X| - |\mathcal A'(X)| &= |X| - |\mathcal A(X)| - 1 \\
	&\geq r^*_M(X) \,.
\end{align*}

Now, $\cl^*(B_i) = \cl^*(A_i)$, which means that $\cl^*_M(X)$ either contains neither $A_i$ nor $B_i$, or both $A_i$ and $B_i$. Thus, Lemma~\ref{lem: rank} implies that
\begin{align*}
	r^*_M(X) &= |\cl^*_M(X)| - |\mathcal A(\cl^*_M(X))| \\
	&= |\cl^*_M(X)| - |\mathcal A'(\cl^*_M(X))| \,.
\end{align*}

Therefore, for all $Y \supseteq X$, we have that $|Y| - |\mathcal A'(Y)| \geq r_M^*(X)$, with equality if $Y = \cl^*_M(X)$. Hence, by Lemma~\ref{lem: rank_general}, 
\begin{align*}
r_N^*(X) &= \min_{Y \supseteq X}\left(|Y| - |\mathcal A'(Y)|\right) \\
&= r^*_M(X) \,.
\end{align*}
\end{proof}

A \emph{maximal presentation} of a transversal matroid $M$ is a presentation in which each set is a cyclic flat of $M^*$. Every transversal matroid has a unique maximal presentation \cite{bon72}. The next lemma shows how to obtain this maximal presentation. It can be seen from repeated application of Lemma~\ref{lem: move between presentation}, taking $B_i = \cl^*(A_i)$, or from \cite[Lemma 3]{bon72}.

\begin{lemma} \label{lem: closure is maximal}
Let $M$ be a transversal matroid with presentation $(A_1, A_2, \ldots, A_r)$ such that $r = r(M)$. Then the maximal presentation of $M$ is $(\cl^*(A_1), \cl^*(A_2), \ldots, \cl^*(A_r))$.
\end{lemma}

Observe that Theorem~\ref{thm: contraction} and Theorem~\ref{thm: construction} require that the rank of the transversal matroid is equal to the number of sets in its presentation. This does not limit their usefulness --- every transversal matroid has such a presentation. Furthermore, if a transversal matroid $M$ has a presentation which does not have $r(M)$ sets, then $M$ necessarily contains a co-loop \cite[Lemma 2.3]{bon15}.

\subsection{Mason's $\alpha$-Criterion}
The basis of the proofs of Theorem~\ref{thm: contraction} and Theorem~\ref{thm: construction} is a necessary and sufficient condition for a matroid to be co-transversal due to Mason \cite{mas72}. For a matroid $M$, this condition works by determining, if $M$ were to be co-transversal, the number of times each set would have to appear in a presentation of its dual.

More precisely, consider a co-transversal matroid $M$ whose dual has maximal presentation $\mathcal A$. Each set of $\mathcal A$ is a cyclic flat $F$ of $M$. By Lemma~\ref{lem: rank}, the number of sets of $\mathcal A$ contained within $F$ is equal to $|F| - r(F)$. Hence, if $F$ properly contains $k$ sets of $\mathcal A$, then the number of times $F$ itself appears in $\mathcal A$ is $|F| - r(F) - k$. This motivates the recursive definition of the $\alpha$-function for a matroid $M$ and a subset $X \subseteq E(M)$ given below:
$$
\alpha_M(X) = |X| - r(X) - \sum_{\substack{F \in \mathcal Z(M) \\ F \subset X}} \alpha_M(F)
$$

Using the ideas discussed above, we can use $\alpha_M$ to construct the maximal presentation of a transversal matroid, as in the next lemma.

\begin{lemma} \label{lem: alpha_construction}
Let $M$ be the dual of a transversal matroid with maximal presentation $\mathcal A = (A_1, A_2, \ldots, A_r)$. If $F \in \mathcal Z(M)$, then there are exactly $\alpha_M(F)$ integers $i \in [r]$ such that $A_i = F$.
\end{lemma}

Lemma~\ref{lem: alpha_construction} immediately implies that, if $M$ is co-transversal, then $\alpha_M(F) \geq 0$ for all $F \in \mathcal Z(M)$. In fact, $\alpha_M(X) \geq 0$ for all $X \subseteq E(M)$, and co-transversal matroids are precisely the matroids for which this is true.

\begin{lemma} \label{lem: alpha}
A matroid $M$ is co-transversal if and only if $\alpha_M\left(X\right) \geq 0$ for all $X \subseteq E(M)$.
\end{lemma}

Proofs of Lemmas~\ref{lem: alpha_construction} and \ref{lem: alpha} can be found in \cite{ing73}, \cite{kun78}, or \cite{bon15}.
	
\section{Single-Element Contractions} \label{contractions}

In this section, we prove Theorems~\ref{thm: contraction} and \ref{thm: construction}. Throughout, let $M$ be the dual of a transversal matroid with presentation $\mathcal A = (A_1, A_2, \ldots, A_r)$ such that $r = r(M)$. Order the sets of $\mathcal A$ such that there exists $k \in [r]$ with $e \in A_i$ for all $i \leq k$ and $e \notin A_j$ for all $j > k$. Let $e \in E(M)$ be the element of $M$ which we are trying to delete. For ease of reading, we define the following notation for a set $X \subseteq E(M)$.
\begin{align*}
\mathcal A(X) &= \{i \in [r] : A_i \subseteq X\} \,, \\
\mathcal A_e(X) &= \{i \in \mathcal A(X) : e \in A_i\} \,.
\end{align*}
Let $G$ be a minimal $(e, \mathcal A)$-presenting graph on vertex set $\mathcal A_e(E(M))$ such that the identity map is an $(e, \mathcal A)$-presenting map of $G$. If $U \subseteq \mathcal A_e(E(M))$, we shall use $G[U]$ to denote the subgraph of $G$ induced by $U$.

\begin{lemma} \label{lem: flat connected}
For all cyclic flats $F \in \mathcal Z(M)$, the graph $G[\mathcal A_e(F)]$ is connected.	
\end{lemma}

\begin{proof}
We show that, for all distinct $i, j \in \mathcal A_e(F)$, there is a path between $i$ and $j$ in $G[\mathcal A_e(F)]$. By the definition of an $(e, \mathcal A)$-presenting graph, the graph $G[\mathcal A_e(\cl_M(A_i \cup A_j))]$ is connected. Furthermore, since $F$ is a flat, we have that $\cl_M(A_i \cup A_j) \subseteq F$, so 
$G[\mathcal A_e(\cl_M(A_i \cup A_j))]$ is a subgraph of $G[\mathcal A_e(F)]$. Therefore, there is a path between $i$ and $j$ in $G[\mathcal A_e(F)]$, completing the proof.
\end{proof}

\begin{lemma} \label{lem: flat deletion}
Let $X \subseteq E(M \backslash e)$. Then
$$
r_{M \backslash e}(X) \leq |X| - |\mathcal A(X)| - \max\{|\mathcal A_e(X \cup \{e\})| - 1, 0\} \,,
$$
with equality if $X$ is a cyclic flat of $M \backslash e$.
\end{lemma} 

\begin{proof}
First, suppose that $|\mathcal A_e(X \cup \{e\})| = 0$. By Lemma~\ref{lem: rank_general},
\begin{align*}
r_{M \backslash e}(X) &= r_M(X)\\
&\leq |X| - |\mathcal A(X)| \,,
\end{align*}
with equality if $X$ is a cyclic flat of $M$. We want to show that equality holds if $X$ is a cyclic flat of $M \backslash e$. There is no $A_i \subseteq X \cup \{e\}$ with $e \in A_i$, so, by Lemma~\ref{lem: flats}, $X \cup \{e\}$ is not a cyclic flat of $M$. Therefore, if $X$ is a cyclic flat of $M \backslash e$, then $X$ is a cyclic flat of $M$, and we have the desired result.

Now suppose $|\mathcal A_e(X \cup \{e\})| \geq 1$. This means that $e \in \cl_M(X)$. Also, note that if $X$ is a cyclic flat of $M \backslash e$, then $X \cup \{e\}$ is a cyclic flat of $M$. Applying Lemma~\ref{lem: rank_general} again, we get
\begin{align*}
r_{M \backslash e}(X) &= r_M(X) \\
&= r_M(X \cup \{e\}) \\
&\leq |X \cup \{e\}| - |\mathcal A(X \cup \{e\})| \\
&= |X \cup \{e\}| - |\mathcal A(X)| - |\mathcal A_e(X \cup \{e\})| \\
&= |X| - |\mathcal A(X)| - (|\mathcal A_e(X \cup \{e\})| - 1) \,,
\end{align*}
with equality if $X \cup \{e\}$ is a cyclic flat of $M$, and thus if $X$ is a cyclic flat of $M \backslash e$. This completes the proof.
\end{proof}

\begin{lemma} \label{lem: deletion alpha}
Suppose $\mathcal A$ is a maximal presentation of $M^*$. Let $F$ be a cyclic flat of $M \backslash e$ such that the graph $G[\mathcal A_e(F \cup \{e\})]$ does not contain a cycle. 
\begin{enumerate}[(i)]
\item If $|\mathcal A_e(F \cup \{e\})| \leq 1$, then $\alpha_{M \backslash e}(F) = |\{i \in \mathcal A(F) : A_i = F\}|$.

\item If $|\mathcal A_e(F \cup \{e\})| > 1$, then $\alpha_{M \backslash e}(F)$ is equal to the number of edges $\{i,j\}$ of $G[\mathcal A_e(F \cup \{e\})]$ such that $\cl_{M \backslash e}(A_i \cup A_j - e) = F$.
\end{enumerate}
\end{lemma}

\begin{proof}
By Lemma~\ref{lem: flat deletion}, if $|\mathcal A_e(F \cup \{e\})| \leq 1$, then
\begin{align*}
\alpha_{M \backslash e}(F) &= |F| - r_{M \backslash e}(F) - \sum_{\substack{F' \in \mathcal Z(M) \\ F' \subset F}} \alpha_M(F') \\
&= |\mathcal A(F)| - \sum_{\substack{F' \in \mathcal Z(M) \\ F' \subset F}} \alpha_M(F') \,,
\end{align*}
and if $|\mathcal A_e(F \cup \{e\})| > 1$, then
$$
\alpha_{M \backslash e}(F) = |\mathcal A(F)| + |\mathcal A_e(F \cup \{e\})| - 1 - \sum_{\substack{F' \in \mathcal Z(M) \\ F' \subset F}} \alpha_M(F') \,.
$$
	
The proof is by induction on the number of cyclic flats properly contained within $F$. First, suppose there are no such cyclic flats. If $|\mathcal A(F \cup \{e\})| \leq 1$, then $\alpha_{M \backslash e}(F) = |\mathcal A(F)|$. For all $i \in \mathcal A(F)$, we have that $A_i$ is a cyclic flat, since $\mathcal A$ is maximal. Therefore, $A_i$ is not properly contained within $F$, so $A_i = F$. Hence, (i) holds.

Otherwise, $|\mathcal A(F \cup \{e\})| > 1$, so $\alpha_{M \backslash e}(F) = |\mathcal A(F)| + |\mathcal A_e(F \cup \{e\})| - 1$. Let $i \in \mathcal A(F)$. We know that $A_i \not\subseteq F$, so $A_i = F$. But now $e \in \cl_M(A_i)$, which contradicts the maximality of $\mathcal A$. Thus, $|\mathcal A(F)| = 0$. By Lemma~\ref{lem: flat connected}, the graph $G[\mathcal A_e(F \cup \{e\})]$ is a tree. Thus, the number of edges in $G[\mathcal A_e(F \cup \{e\})]$ is
$$
|\mathcal A_e(F \cup \{e\})| - 1 = \alpha_{M \backslash e}(F) \,.
$$
Furthermore, for all $i, j \in \mathcal A_e(F \cup \{e\})$, we have that $\cl_{M \backslash e}(A_i \cup A_j - e) = F$, since $F$ does not properly contain any cyclic flats. Thus, (ii) holds.

Now, suppose that the result is true for all cyclic flats $F' \subset F$. Therefore,
$$
\sum_{\substack{F' \in \mathcal Z(M) \\ F' \subset F}} \alpha_{M \backslash e}(F') = |\{i \in \mathcal A(F) : A_i \subset F\}| + k \,,
$$
where $k$ is the number of edges $\{i,j\}$ of $G[\mathcal A_e(F \cup \{e\})]$ such that $\cl_{M \backslash e}(A_i \cup A_j - e) \subset F$.

If $|\mathcal A_e(F \cup \{e\})| \leq 1$, then $k = 0$. Thus,
\begin{align*}
\alpha_{M \backslash e}(F) &= |\mathcal A(F)| - |\{i \in \mathcal A(F) : A_i \subset F\}| \\
&= |\{i \in \mathcal A(F) : A_i = F\}| \,,
\end{align*}
and (i) holds. Finally, suppose $|\mathcal A_e(F \cup \{e\})| > 1$. As before, there is no $i \in \mathcal A(F)$ such that $A_i = F$, since then $e \in \cl_M(A_i)$. Therefore,
\begin{align*}
\alpha_{M \backslash e}(F) &= |\mathcal A(F)| + |\mathcal A_e(F \cup \{e\})| - 1 - |\{i \in \mathcal A(F) : A_i \subset F\}| - k \\
&= |\mathcal A_e(F \cup \{e\})| - 1 - k \,.
\end{align*}
Since $|\mathcal A_e(F \cup \{e\})| - 1$ is the number of edges in $G[\mathcal A_e(F \cup \{e\})]$, this implies that (ii) holds, and completes the proof.
\end{proof}

The next three lemmas are concerned with showing that Theorem~\ref{thm: contraction} is true in the case where $\mathcal A$ is maximal.

\begin{lemma} \label{lem: tree sufficient}
Suppose $\mathcal A$ is a maximal presentation of $M^*$. If $G$ is a tree with edge set
$$
E(G) = \{\{i_1,j_1\},\{i_2,j_2\},\ldots,\{i_{k-1},j_{k-1}\}\}
$$
then
$$
(\cl_{M\backslash e}(A_{i_1} \cup A_{j_1} - e), \cl_{M\backslash e}(A_{i_2} \cup A_{j_2} - e), \ldots, \cl_{M \backslash e}(A_{i_{k-1}} \cup A_{j_{k-1}} - e), A_{k+1}, A_{k+2},\ldots,A_r)
$$
is a presentation of $(M \backslash e)^*$.
\end{lemma}

\begin{proof}
To show that $M \backslash e$ is co-transversal, we need to show that $\alpha_{M \backslash e}(X) \geq 0$ for all $X \subseteq E(M)$. By Lemma~\ref{lem: deletion alpha}, this is true if $X$ is a cyclic flat. So assume $X$ is not a cyclic flat. If $|\mathcal A(F \cup \{e\})| \leq 1$, then Lemma~\ref{lem: deletion alpha} implies that
$$
\sum_{\substack{F \in \mathcal Z(M) \\ F \subset X}} \alpha_{M \backslash e}(F) = |\mathcal A(X)| \,.
$$
Now, by Lemma~\ref{lem: rank_general},
\begin{align*}
\alpha_{M \backslash e}(X) &= |X| - r_{M \backslash e}(X) - |\mathcal A(X)| \\
&= |X| - r_M(X) - |\mathcal A(X)| \\
&\geq |X| - (|X| - |\mathcal A(X)|) - |\mathcal A(X)| \\
&= 0 \,.
\end{align*}

Otherwise, $|\mathcal A(F \cup \{e\})| > 1$. Since $G[\mathcal A_e(X \cup \{e\})]$ does not contain a cycle, it has at most $|\mathcal A_e(X \cup \{e\})| - 1$ edges. Thus, applying Lemma~\ref{lem: deletion alpha}, we get
\begin{align*}
\sum_{\substack{F \in \mathcal Z(M) \\ F \subset X}} \alpha_{M \backslash e}(F) &\leq |\mathcal A(X)| + |\mathcal A_e(X \cup \{e\})| - 1 \\
&= |\mathcal A(X \cup \{e\})| - 1 \,,
\end{align*}
Therefore, using the fact that $e \in \cl_M(X)$ and Lemma~\ref{lem: rank_general},
\begin{align*}
\alpha_{M \backslash e}(X) &\geq |X| - r_{M \backslash e}(X) - |\mathcal A(X \cup \{e\})| + 1 \\
&=	|X| - r_M(X \cup \{e\}) - |\mathcal A(X \cup \{e\})| + 1 \\
&\geq |X| - (|X \cup \{e\}| - |\mathcal A(X \cup \{e\})|) - |\mathcal A(X \cup \{e\})| + 1 \\
&= 0 \,.
\end{align*}
Thus, $\alpha_{M \backslash e}(X) \geq 0$, so $M \backslash e$ is co-transversal, and Lemma~\ref{lem: alpha_construction} implies that $(M \backslash e)^*$ has the desired presentation.
\end{proof}

\begin{lemma} \label{lem: tree necessary}
Suppose $\mathcal A$ is a maximal presentation of $M^*$. If $G$ has a cycle, then $M \backslash e$ is not co-transversal.
\end{lemma}

\begin{proof}
Let $F$ be a cyclic flat of $M \backslash e$ such that $G[\mathcal A_e(F)]$ contains a cycle $C$. Furthermore, suppose there is no cyclic flat $F' \subset F$ of $M \backslash e$ such that $G[\mathcal A_e(F')]$ contains a cycle. First, we show there is no edge $\{i,j\}$ of $G[\mathcal A_e(F)]$ such that $\cl_{M \backslash e}(A_i \cup A_j - e) = F$. Suppose such an edge $\{i,j\}$ exists, and let $G'$ be the graph obtained by deleting $\{i,j\}$ from $G$. We shall show that $G'$ is an $(e, \mathcal A)$-presenting graph. Take distinct $i_0, j_0 \in \mathcal A_e(E(M))$, and let $X = \cl_M(A_{i_0} \cup A_{j_0})$. If either $A_i \not\subseteq X$ or $A_j \not\subseteq X$, then $G'[\mathcal A_e(X)] = G[\mathcal A_e(X)]$, and is connected. On the other hand, if $A_i \subseteq X$ and $A_j \subseteq X$, then $F = \cl_{M \backslash e}(A_i \cup A_j - e) \subseteq X$. In particular, every vertex of the cycle $C$ is contained in $G'[\mathcal A_e(X)]$. Therefore, there is still a path between $i$ and $j$ in $G'[\mathcal A_e(X)]$, so this graph is connected. Thus, $G'$ is $(e, \mathcal A)$-presenting. But this means that $G$ is not a minimal $(e, \mathcal A)$-presenting graph, a contradiction. Thus, no such edge $\{i,j\}$ exists.

In particular, this means that, for every edge $\{i',j'\}$ of $G[\mathcal A_e(F \cup \{e\})]$, we have that $\cl_{M \backslash e}(A_{i'} \cup A_{j'} - e) \subset F$. Furthermore, there is no $i \in \mathcal A(F)$ such that $A_i = F$, as then $e \in \cl_M(A_i)$. So, by Lemma~\ref{lem: deletion alpha},
$$
\sum_{\substack{F' \in \mathcal Z(M) \\ F' \subset F}} \alpha_{M \backslash e}(F') = |\mathcal A(F)| + k \,,
$$
where $k$ is the number of edges of $G[\mathcal A_e(F \cup \{e\})]$.

Therefore, by Lemma~\ref{lem: flat deletion},
\begin{align*}
\alpha_{M \backslash e}(F) &= |\mathcal A(F)| + |\mathcal A_e(F \cup \{e\})| - 1 - |\mathcal A(F)| - k \\
&= |\mathcal A_e(F \cup \{e\})| - 1 - k \,.
\end{align*}
$G[\mathcal A(F \cup \{e\})]$ contains a cycle and, by Lemma~\ref{lem: flat connected}, is connected. Thus, $k \geq |\mathcal A_e(F \cup \{e\})|$, and so $\alpha_{M \backslash e}(F) < 0$. Hence, $M \backslash e$ is not co-transversal.
\end{proof}

For the remainder of this section, we show that Theorems~\ref{thm: contraction} and \ref{thm: construction} follow from Lemmas~\ref{lem: tree sufficient} and \ref{lem: tree necessary} in the case where $\mathcal A$ is not a maximal presentation.

\begin{lemma} \label{lem: extend presentation}
Let $i \in [r]$, and let
$$
\mathcal A' = (A_1, A_2, \ldots, A_{i-1}, \cl_M(A_i), A_{i+1}, A_{i+2}, \ldots, A_r) \,.
$$
Let $G'$ be a graph defined as follows.
\begin{enumerate} [(i)]
\item If $e \notin \cl_M(A_i) - A_i$, then $G' = G$.
\item If $e \in \cl_M(A_i) - A_i$, then $G'$ is constructed from $G$ by adding a vertex $i$, and adding an edge from $i$ to exactly one $j \in \mathcal A_e(\cl_M(A_i))$.
\end{enumerate}
Then $G'$ is a minimal $(e, \mathcal A')$-presenting graph.
\end{lemma}

\begin{proof}
First, suppose that $e \notin \cl_M(A_i) - A_i$, so $G' = G$. This means that $\mathcal A_e(E(M)) = \mathcal A'_e(E(M))$. Furthermore, for all distinct $i', j' \in \mathcal A'_e(E(M))$, we have that 
$$
\mathcal A'_e(\cl_M(A_{i'} \cup A_{j'})) = \mathcal A_e(\cl_M(A_{i'} \cup A_{j'})) \,.
$$ 
Thus, a graph is $(e, \mathcal A')$-presenting if and only if it is $(e, \mathcal A)$-presenting, so $G'$ is a minimal $(e, \mathcal A')$-presenting graph.

Otherwise, suppose $e \in \cl_M(A_i) - A_i$, and that $G'$ is constructed as described. Observe that $\mathcal A'_e(E(M)) = \mathcal A_e(E(M)) \cup \{i\}$, so $G'$ has the correct vertex set. Additionally, since $e \in \cl_M(A_i)$, Lemma~\ref{lem: flats} implies that there does exist $j \in \mathcal A_e(\cl_M(A_i))$, as required by the construction of $G'$.
	
We show that $G'$ is $(e, \mathcal A')$-presenting. Take distinct $i', j' \in \mathcal A'_e(E(M))$, and let $X = \cl_{M}(A_{i'} \cup A_{j'})$. If $\cl_M(A_i) \not\subseteq X$, then $G'[\mathcal A'_e(X)] = G[\mathcal A_e(X)]$, so $G'[\mathcal A'_e(X)]$ is also connected. Otherwise, $\cl_M(A_i) \subseteq X$. Since $j \in \mathcal A_e(\cl_M(A_i))$, we have that $j$ is a vertex of both $G[\mathcal A_e(X)]$ and $G'[\mathcal A'_e(X)]$. Thus, $G'[\mathcal A'_e(X)]$ can be constructed from $G[\mathcal A_e(X)]$ by adding the vertex $i$ and the edge $\{i,j\}$. Hence, $G'[\mathcal A'_e(X)]$ is connected, so $G'$ is $(e, \mathcal A')$-presenting.

Now we show that $G'$ is a minimal $(e, \mathcal A')$-presenting graph. 
Let $H'$ be a graph obtained from $G'$ by deleting an edge, say $\{i',j'\}$. If $\{i',j'\} = \{i,j\}$, then the vertex $i$ has no incident edges in $H'$, so this graph is disconnected. Hence, by Lemma~\ref{lem: flat connected}, $H'$ is not $(e, \mathcal A')$-presenting. Otherwise, $\{i',j'\}$ is also an edge of $G$, so let $H$ be the graph obtained from $G$ by deleting $\{i',j'\}$. $H$ is not $(e, \mathcal A)$-presenting, so there exist distinct $i'',j'' \in \mathcal A_e(E(M))$ such that $H[\mathcal A_e(\cl_M(A_{i''} \cup A_{j''}))]$ is disconnected. Clearly, adding a new vertex $i$ and one edge to $i$ cannot connect this graph, so $H'[\mathcal A_e(\cl_M(A_{i''} \cup A_{j''}))]$ is also disconnected. Thus, $H'$ is not $(e, \mathcal A)$-presenting, which completes the proof.
\end{proof}

\begin{lemma} \label{lem: simplify presentation}
Suppose $G$ is a tree with edge set
$$
\{\{i_1,j_1\},\{i_2,j_2\},\ldots,\{i_{k-1},j_{k-1}\}\} \,.
$$
Define $B_1, B_2, \ldots, B_{r-1}$ by 
\begin{enumerate} [(i)]
\item $B_m = A_{i_m} \cup A_{j_m} - e$ if $1 \leq m \leq k-1$, and 
\item $B_m = A_{m+1}$ if $k \leq m \leq r-1$. 
\end{enumerate}
Choose $J \subseteq [r-1]$. Suppose $M \backslash e$ has presentation $\mathcal C = (C_1, C_2, \ldots, C_{r-1})$, where $C_m = B_m$ if $m \notin J$, and $C_m = \cl_{M \backslash e}(B_m)$ if $m \in J$. Let $j \in J$. Then
$$
\mathcal C' = (C_1, C_2, \ldots, C_{j-1}, B_j, C_{j+1}, C_{j+2}, \ldots, C_{r-1})
$$
is a presentation of $M \backslash e$.
\end{lemma}

\begin{proof}	
Let $X \supseteq B_j$ such that $C_j \not\subseteq X$. We show that $r_{M \backslash e}(X) < |X| - |\mathcal C(X)|$, and the result follows by Lemma~\ref{lem: move between presentation}. Define
\begin{align*}
	\mathcal C^-(X) &= \{i \in \mathcal C(X) : i < k\} \\
	\mathcal C^+(X) &= \{i \in \mathcal C(X) : i \geq k\} \,.
\end{align*}

Let $\ell \in \mathcal C(X)$. If $\ell \geq k$, then
$$
B_\ell = A_{\ell + 1} \subseteq C_{\ell} \subseteq X \,,
$$
so $|\mathcal C^+(X)| \leq |\mathcal A(X)|$. Otherwise, $\ell < k$, and
$$
B_\ell = A_{i_\ell} \cup A_{j_\ell} - e \subseteq C_\ell \subseteq X \,.
$$
In particular, every element of $\mathcal C^-(X)$ correponds to an edge of $G[\mathcal A_e(X)]$. Since $G$ is a tree, $|\mathcal C^-(X)| \leq \max\{|\mathcal A_e(X \cup \{e\})| - 1, 0\}$.

We have that $B_j \subseteq X$ and $C_j \not\subseteq X$. Therefore, if $j \geq k$, then $|\mathcal C^+(X)| < |\mathcal A(X)|$, and if $j < k$, then $|\mathcal C^-(X)| < \max\{|\mathcal A_e(X \cup \{e\})| - 1, 0\}$. By Lemma~\ref{lem: flat deletion},
\begin{align*}
	r_{M \backslash e}(X) &\leq |X| - |\mathcal A(X)| - \max\{\mathcal A_e(X \cup \{e\}) - 1, 0\} \\
	&< |X| - |\mathcal C^+(X)| - |\mathcal C^-(X)| \\
	&= |X| - |\mathcal C(X)| \,,
\end{align*}
which completes the proof.
\end{proof}

\begin{proof}[Proof of Theorems~\ref{thm: contraction} and \ref{thm: construction}]
Let $\mathcal A^+$ be a maximal presentation of $M$. Repeatedly apply Lemma~\ref{lem: extend presentation} to construct a minimal $(e, \mathcal A^+)$-presenting graph $G^+$. Note that $G^+$ is a tree if and only if $G$ is a tree, so Theorem~\ref{thm: contraction} follows from Lemmas~\ref{lem: tree sufficient} and \ref{lem: tree necessary}.

To prove Theorem~\ref{thm: construction}, suppose $G$ is a tree with edge set
$$
\{\{i_1,j_1\},\{i_2,j_2\},\ldots,\{i_{k-1},j_{k-1}\}\} \,.
$$
Let $\{i,j\}$ be an edge of $G^+$ which is not an edge of $G$. By the construction of $G^+$, either $e \notin A_i$ and $A_j \subseteq \cl_M(A_i)$, or vice versa. Assume the former without loss of generality. Then
$$
\cl_{M \backslash e}(A_i \cup A_j - e) = \cl_{M \backslash e}(A_i) \,.
$$
Therefore, by Lemma~\ref{lem: tree sufficient},
\begin{align*}
(\cl_{M \backslash e}(A_{i_1} \cup A_{j_1} - e), \cl_{M \backslash e}(A_{i_2} \cup A_{j_2} - e), \ldots, \cl_{M \backslash e}(A_{i_{k-1}} \cup A_{j_{k-1}} - e), \\
\cl_{M \backslash e}(A_{k+1}), \cl_{M \backslash e}(A_{k+2}), \ldots, \cl_{M \backslash e}(A_r))
\end{align*}
is a presentation of $M \backslash e$. In particular, this is a presentation of the form described in Lemma~\ref{lem: simplify presentation} with $J = [r-1]$. Repeated application of Lemma~\ref{lem: simplify presentation} now proves Theorem~\ref{thm: construction}.
\end{proof}

\section{Path-Circular Matroids} \label{minorclosed}

In this section, we prove that path-circular matroids are closed under minors. Throughout this section, let $G$ be a simple graph, and let $\mathcal P$ be a path-circular collection of $G$. Let $p = u_1,u_2,\ldots,u_k$ be a path of $\mathcal P$, and suppose the vertex set of $G$ is 
$$
V(G) = \{u_1,u_2,\ldots,u_k,v_{k+1},v_{k+2},\ldots,v_r\} \,.
$$  
By the definition of a path-circular matroid, 
$$
\mathcal A = (N(u_1),N(u_2),\ldots,N(u_k),N(v_{k+1}),N(v_{k+2}),\ldots,N(v_r))
$$ 
is a presentation of $M(\mathcal P)$. 

We will show that $M(\mathcal P)/p$ is path-circular using Theorem~\ref{thm: construction}. However, applying Theorem~\ref{thm: construction} requires that $|\mathcal A| = r(M(\mathcal P))$, which may not be true if $M(\mathcal P)$ contains a coloop. However, we can freely add a coloop to a path-circular matroid by adding a new vertex to the graph, and adding a new path which contains only that vertex. Hence, if $M(\mathcal P)$ contains co-loops, we can delete the coloops, determine $M(\mathcal P) / p$, then add the coloops back in. Thus, we may assume that $M(\mathcal P)$ does not contain any coloops, and so $|\mathcal A| = r(M(\mathcal P))$.

Now, we describe how to construct a graph $G'$ and set of paths $\mathcal P'$ of $G'$ such that $M(\mathcal P) / p = M(\mathcal P')$. First, construct a new graph from $G$ as follows. For all $i \in \{1,2,\ldots,k-1\}$, subdivide the edge $\{u_i,u_{i+1}\}$ by a new vertex $u_{i,i+1}$, and, if $i > 1$, add an edge to $u_{i-1,i}$. Add edges between $u_{1,2}$ and every other vertex adjacent to $u_1$, and add edges between $u_{k-1,k}$ and every other vertex adjacent to $u_k$. Finally, delete the vertices $u_1, u_2, \ldots, u_k$. Let $G'$ be the resulting graph.

Next, we construct paths $\mathcal P'$ of $G'$ from the paths in $\mathcal P - p$. Consider each path $p' \in \mathcal P - p$. If $p'$ contains $u_1$, then, for all $u_i \in p'$ with $i \neq k$, replace $u_i$ with $u_{i,i+1}$. If $p'$ contains $u_k$ and does not contain $u_1$, then, for all $u_i \in p'$ with $i \neq 1$, replace $u_i$ with $u_{i-1,i}$. Finally, if the resulting path starts and ends with the same vertex, then delete either the start vertex or the end vertex. An example is shown in Figure~\ref{fig: contraction}.

\begin{figure}
\centering
\begin{subfigure} {\textwidth}
\centering
\begin{tikzpicture}
\coordinate (a1) at (30:3);
\coordinate (a2) at (90:3);
\coordinate (a3) at (150:3);
\coordinate (a4) at (210:3);
\coordinate (a5) at (270:3);
\coordinate (a6) at (330:3);
\coordinate (a7) at ($(a1) + (3,0)$);
\coordinate (a8) at ($(a6) + (3,0)$);

\draw[thick] (a1) .. controls ($(a1) + (0.4,1.3)$) and ($(a1) + (1.1,0.7)$) .. (a1);
\draw[thick] (a8) .. controls ($(a8) + (0.6,-1.25)$) and ($(a8) + (1.25,-0.6)$) .. (a8);

\draw[fill=white] (a1) circle (8pt);
\draw[fill=white] (a2) circle (8pt);
\draw[fill=white] (a3) circle (8pt);
\draw[fill=white] (a4) circle (8pt);
\draw[fill=white] (a5) circle (8pt);
\draw[fill=white] (a6) circle (8pt);
\draw[fill=white] (a7) circle (8pt);
\draw[fill=white] (a8) circle (8pt);

\draw[very thick, dashed] (32:2.8) arc (32:210:2.8); 
\draw[thick] (210:3.2) arc (210:270:3.2);
\draw[thick] (90:2.93) arc (90:270:2.93);
\draw[thick] (210:3.06) arc (210:327:3.06);
\draw[thick] ($(a1)+(0.1,0.1)$) -- ($(a1)!0.8!(a7)+(0.2,0.1)$) .. controls ($(a7)+(0.2,0.1)$) .. ($(a7)!0.2!(a8)+(0.2,0)$) -- ($(a8)+(0.2,0)$);
\draw[thick] ($(a1) + (0.1,-0.05)$) -- ($(a7) + (-0.05,-0.05)$);
\draw[thick] ($(a7)+(0,-0.15)$) -- ($(a7)!0.8!(a8)$) .. controls ($(a8)+(0,-0.1)$) .. ($(a8)!0.2!(a6)+(0,-0.1)$) -- ($(a6)+(0.1,-0.1)$);
\draw[thick] ($(a1)+(-0.1,-0.05)$) -- ($(a6)+(-0.1,0.05)$);
\draw[thick] (32:3.2) arc (32:90:3.2);
\draw[thick] (270:2.8) arc (270:329:2.8);
\end{tikzpicture}
\subcaption{}
\end{subfigure}

\begin{subfigure} {\textwidth}
\centering
\begin{tikzpicture}
\coordinate (a1) at (30:3);
\coordinate (a2) at (90:3);
\coordinate (a3) at (150:3);
\coordinate (a4) at (210:3);
\coordinate (a5) at (270:3);
\coordinate (a6) at (330:3);
\coordinate (a7) at ($(a1) + (3,0)$);
\coordinate (a8) at ($(a6) + (3,0)$);

\coordinate (b1) at (60:3);
\coordinate (b2) at (120:3);
\coordinate (b3) at (180:3);

\draw[thick] (b1) .. controls ($(b1) + (0.4,1.3)$) and ($(b1) + (1.1,0.7)$) .. (b1);
\draw[thick] (a8) .. controls ($(a8) + (0.6,-1.25)$) and ($(a8) + (1.25,-0.6)$) .. (a8);

\draw[fill=white] (a5) circle (8pt);
\draw[fill=white] (a6) circle (8pt);
\draw[fill=white] (a7) circle (8pt);
\draw[fill=white] (a8) circle (8pt);

\draw[fill=lightgray] (b1) circle (8pt);
\draw[fill=lightgray] (b2) circle (8pt);
\draw[fill=lightgray] (b3) circle (8pt);

\draw[thick] (180:3.2) arc (180:270:3.2);
\draw[thick] (60:2.93) arc (60:270:2.93);
\draw[thick] (180:3.06) arc (180:327:3.06);
\draw[thick] ($(b1)+(0.1,0.1)$) -- ($(b1)!0.8!(a7)+(0.2,0.1)$) .. controls ($(a7)+(0.2,0.1)$) .. ($(a7)!0.2!(a8)+(0.2,0)$) -- ($(a8)+(0.2,0)$);
\draw[thick] ($(b1) + (0.1,-0.05)$) -- ($(a7) + (-0.05,-0.05)$);
\draw[thick] ($(a7)+(0,-0.15)$) -- ($(a7)!0.8!(a8)$) .. controls ($(a8)+(0,-0.1)$) .. ($(a8)!0.2!(a6)+(0,-0.1)$) -- ($(a6)+(0.1,-0.1)$);
\draw[thick] ($(b1)+(-0.1,-0.1)$) -- ($(a6)+(-0.1,0.05)$);
\draw[thick] (60:3.2) arc (60:120:3.2);
\draw[thick] (270:2.8) arc (270:329:2.8);
\end{tikzpicture}
\subcaption{}
\end{subfigure}
\caption{A path-circular matroid (A), and the contraction of the dashed path (B). The new vertices are shown in grey.} \label{fig: contraction}
\end{figure}
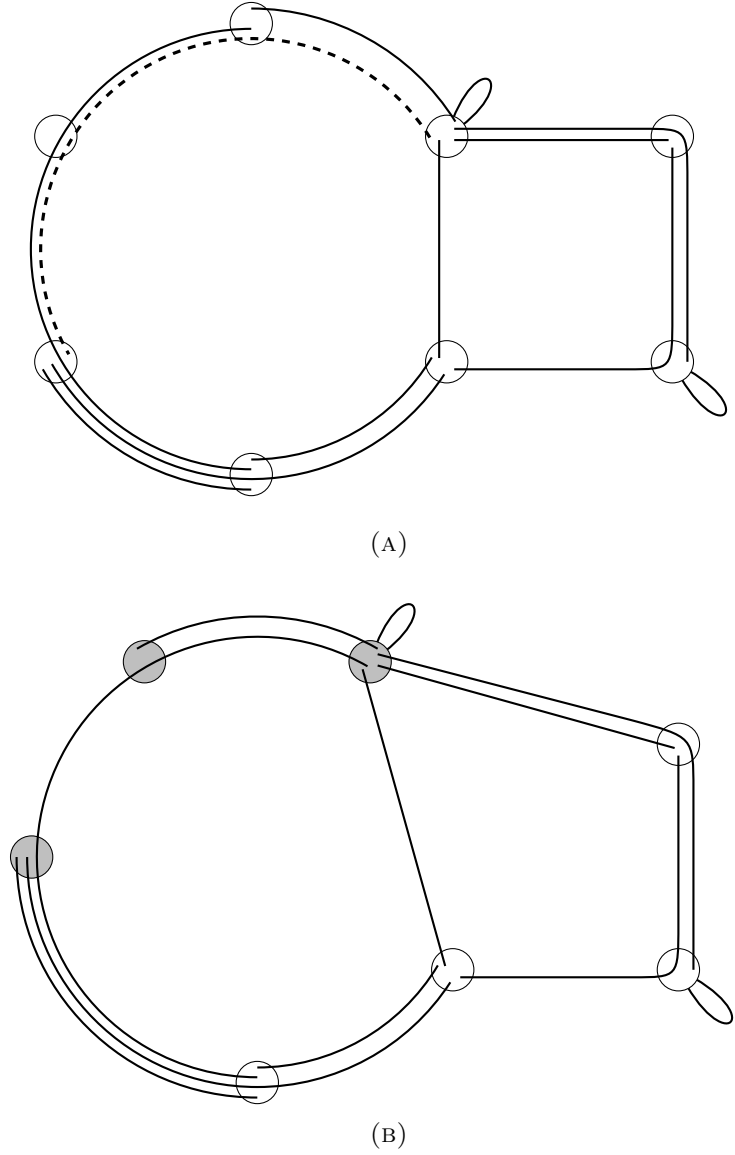

\begin{lemma} \label{lem: p' path circular}
$\mathcal P'$ is a path-circular collection of $G'$.
\end{lemma}

\begin{proof}
Let $p' = u_1',u_2',\ldots,u'_{k'}$ be a path of $\mathcal P'$, and let $p''$ be the corresponding path of $\mathcal P - p$. Let $i \in \{2,3,\ldots,k'-1\}$. We need to show that
\begin{enumerate}[(i)]
	\item the degree of $u'_i$ in $G'$ is $2$, and
	\item every path of $\mathcal P'$ which contains $u'_i$ also contains either $u'_1$ or $u'_{k'}$.
\end{enumerate}
	
If $u'_i \in \{v_{k+1},v_{k+2},\ldots,v_r\}$, then $u'_i$ is also a vertex of $p''$. Therefore, $u'_i$ has degree $2$ in $G$, so also degree $2$ in $G'$. Otherwise $u'_i = u_{j,j+1}$ for some $j \in \{1,2,\ldots,k-1\}$. If $j \neq 1$ and $j \neq k-1$, then $u_{j,j+1}$ has degree $2$ in $G'$. Without loss of generality, suppose $j = 1$. Since $i \notin \{1,k'\}$, one of $u'_{i-1}$ and $u'_{i+1}$ is $u_{2,3}$ and the other is a vertex $v \in \{v_{k+1},v_{k+2},\ldots,v_r\}$. This implies that the path $P''$ has a sub-path $v,u_1,u_2$. Therefore, $u_1$ has degree $2$ in $G$, so $u_{1,2}$ has degree $2$ in $G'$. This establishes (i).

Now, suppose $q \in \mathcal P'$ contains $u_i'$. Let $q'$ be the corresponding path of $\mathcal P - p$. If $u'_i$ is an element of $\{v_{k+1},v_{k+2},\ldots,v_r\}$, then $u_i'$ is a vertex of both $q'$ and $p''$. Otherwise, $u_i' = u_{j,j+1}$ for some $j \in \{1,2,\ldots,k-1\}$. This implies that $q'$ contains at least one of $u_j$ and $u_{j+1}$. Since $u_{j,j+1}$ is not an end vertex of $p'$, we also have that both $u_j$ and $u_{j+1}$ are vertices of $p''$. Therefore, either $u_j$ or $u_{j+1}$ is a vertex of both $q'$ and $p''$. In either case, the paths $q'$ and $p''$ contain a vertex in common, which means that $q'$ also contains an end vertex of $p''$. Let this vertex be $u''_1$. If $u''_1 \in \{v_{k+1},v_{k+2},\ldots,v_r\}$, then $u''_1$ is also an end vertex of $p'$ and a vertex of $q$, and we have the desired result. Otherwise, $u''_1 = u_j$, for some $j \in [k]$. This implies that either $u_{j-1,j}$ or $u_{j,j+1}$ is an end vertex of $p'$. Furthermore, since $q'$ contains $u_j$, the path $q$ contains $u_{j-1,j}$ (unless $j=1$) and $u_{j,j+1}$ (unless $j = k$). Thus, $q$ contains an end vertex of $p'$, completing the proof.
\end{proof}

\begin{lemma} \label{lem: path minimal graph}
The path $p$ is a minimal $(p, \mathcal A)$-presenting graph.
\end{lemma}

\begin{proof}
Let $i_1, i_2, i_3 \in [k]$ such that $i_1 < i_2 < i_3$. We show that $N(u_{i_2}) \subseteq N(u_{i_1}) \cup N(u_{i_3})$. Let $p' \in N(u_{i_2})$. Since $\mathcal P$ is a path-circular collection, the path $p'$ either contains $u_1$ or $u_k$. In the former case, $p'$ contains the vertices $u_1,u_2,\ldots,u_i$, since the degree of each of these vertices is $2$ in $G$. In particular, $p'$ contains the vertex $u_{i_1}$, so $p' \in N(u_{i_1})$. Symmetrically, in the latter case, $p' \in N(u_{i_3})$. Therefore, $p' \in N(u_{i_1}) \cup N(u_{j_3})$, so $N(u_{i_2}) \subseteq N(u_{i_1}) \cup N(u_{i_3})$.

Now, consider $i, j \in [k]$ with $i < j$. Let $i'$ be the minimum element of $[k]$ such that 
$$
N(u_{i'}) \subseteq \cl^*_{M(\mathcal P)}(N(u_i) \cup N(u_j)) \,,
$$ 
and let $j'$ be the maximum element of $[k]$ such that 
$$
N(u_{j'}) \subseteq \cl^*_{M(\mathcal P)}(N(u_i) \cup N(u_j)) \,.
$$ 
We have already shown that, for all $x \in \{i',i'+1,\ldots,j'\}$,
$$
N(u_x) \subseteq N(u_{i'}) \cup N(u_{j'}) \subseteq \cl^*_{M(\mathcal P)}(N(u_i) \cup N(u_j)) \,.
$$
Therefore, the subgraph of $P$ induced by the vertices $\{u \in V(p) : N(u) \subseteq \cl^*_{M(\mathcal P)}(N(u_i) \cup N(u_j))\}$ is the path $u_{i'},u_{i'+1},\ldots,u_{j'}$. This is connected, so $p$ is a $(p, \mathcal A)$-presenting graph. Furthermore, $p$ is a tree, so $p$ is a minimal $(p, \mathcal A)$-presenting graph by Lemma~\ref{lem: flat connected}.
\end{proof}

\begin{lemma} \label{lem: contraction}
$M(\mathcal P) / p = M(\mathcal P')$.	
\end{lemma}

\begin{proof}
Let $v$ be a vertex of $G'$. Suppose $v = u_{i,i+1}$ for some $i \in \{1,2,\ldots,k-1\}$. Then each path of $\mathcal P'$ which contains $u_{i,i+1}$ corresponds to a path of $\mathcal P - p$ which contains either $u_i$ or $u_{i+1}$. Thus, the paths of $\mathcal P'$ which contain $v$ are the paths which correspond to $N(u_i) \cup N(u_{i+1}) - p$. Otherwise, each path of $\mathcal P'$ which contains $v$ corresponds to a path of $\mathcal P - p$ which also contains $v$. Therefore, the paths of $\mathcal P'$ which contain $v$ are the paths which correspond to $N_G(v) - p = N_G(v)$. Hence,
$$
(N(u_1) \cup N(u_2) - P, N(u_2) \cup N(u_3) - P, \ldots, N(u_{k-1}) \cup N(u_k) - P, N(v_{k+1}), N(v_{k+2}), \ldots, N(v_r))
$$
is a presentation of $M(\mathcal P')$. By Lemma~\ref{lem: path minimal graph} and Theorem~\ref{thm: construction}, this is also a presentation of $M(\mathcal P) / p$.
\end{proof}

\begin{proof}[Proof of Theorem~\ref{thm: minors}]
By Lemma~\ref{lem: deletion presentation}, $M(\mathcal P) \backslash p = M(\mathcal P - p)$, and, by Lemma~\ref{lem: contraction}, $M(\mathcal P) / p = M(\mathcal P')$.
\end{proof}

\printbibliography
\end{document}